\documentclass[a4paper,reqno,11pt]{amsart}

\usepackage[utf8]{inputenc,xcolor}  
\usepackage{marginnote}
\usepackage{geometry}
\usepackage{float} 
\usepackage[in]{fullpage}

\newcommand{\ds}{\displaystyle}

\newcommand{\tensor}{\otimes}

\newcommand{\op}{\mathcal}

\newcommand{\cdc}{,\dots,}

\usepackage{amsmath}%
\usepackage{amsthm}
\usepackage{amsfonts}%
\usepackage{amssymb}%
\usepackage{graphicx}
\usepackage{xy,amsthm,enumerate,xypic,array}  

\input{xy}
\xyoption{all}

\numberwithin{equation}{section}

\newtheorem{theorem}{Theorem}[section]
\theoremstyle{plain}

\newtheorem{corollary}[theorem]{Corollary}
\newtheorem{lemma}[theorem]{Lemma}
\newtheorem{proposition}[theorem]{Proposition}

\theoremstyle{definition}
\newtheorem{definition}[theorem]{Definition}
\newtheorem{example}[theorem]{Example}

\allowdisplaybreaks[2]

\addtocounter{MaxMatrixCols}{2}

\begin{document}

\author{Ralph M. Kaufmann and Benjamin C. Ward}


\title{Schwarz Modular Operads Revisited: $\op{SM}=\op{S}\circ\op{M}$.}
\maketitle
\begin{abstract}We prove that the Feynman category encoding Schwarz's variant of modular operads is Koszul.  Our proof uses a generalization of the theory of distributive laws to the groupoid colored setting.
	\end{abstract}


		\tableofcontents

\section{Introduction}
To begin, let us revisit an elementary observation about gluing surfaces along their boundaries.  Let $\Sigma=\Sigma_{g,n}$ be a compact, oriented and connected surface of genus $g$ with $n$ numbered boundary components.  Suppose that our surface is stable, i.e.\ $2g+n\geq 3$.  The surface $\Sigma$ can be cut into pairs of pants (i.e.\ surfaces of the form $\Sigma_{0,3}$) or conversely assembled from pairs of pants by gluing along boundary components.  These gluings come in two types, self-gluings, which glue two boundary components on the same surface, and non-self gluings which glue two boundary components on different surfaces.

Now suppose, by comparison, that we allow our surfaces to be disconnected.  For convenience assume that each connected component is stable.  No longer can every surface be formed by gluing pairs of pants along their boundary, since that always results in a connected surface.  However one can simply allow ``mergers'' which take the disjoint union of two surfaces to form a new, necessarily disconnected surface: $(\Sigma_{g,n}, \Sigma_{h,m}) \mapsto \Sigma_{g,n}\sqcup\Sigma_{h,m}$.  With mergers, it is again the case that all surfaces (now including disconnected) can be formed from pairs of pants.  However, in this case we need not perform any non-self gluings because each non-self gluing can be replaced with a merger of the two surfaces to be glued, followed by a self gluing of the two relevant boundary components.  Since these two boundary components belong to a single surface after the merger, the gluing which connects them is a self-gluing.

This seemingly trivial distinction has a significant impact on the quadratic presentation, and hence Koszul duality theory, of two operads which encode these types of surface gluings.  The algebras over these two operads are called ``modular operads'' and ``Schwarz modular operads'' respectively.  The notion of modular operads, introduced by Getzler and Kapranov in \cite{GeK2} is a generalization of operads allowing composition along any {\it connected} graph.  Contemporaneously, Schwarz introduced, in \cite{S}, a variation of this notion which allows for composition along any graph, including disconnected graphs.  We wish to remark here that we only consider gluing of boundary components and disjoint union; also of interest is the connected sum operation, for which we refer to \cite{CB22} for more detail.

Both modular and Schwarz modular operads are themselves encoded by quadratic operads, call them $\op{MO}$ and $\op{SM}$.  As in the discussion above, the operad $\op{MO}$ is generated by non-self gluings and self-gluings.  The operad $\op{SM}$ is generated by mergers and self-gluings.  There is a natural map of operads between them $\op{MO}\to\op{SM}$, but crucially it is not a quadratic map since the non-self gluings are generators in the source, but not in the target.

The operad $\op{MO}$ was shown to be Koszul in \cite{WardMP}, and we generalized this result in \cite{KFC} to a large family of Feynman categories encoding operad-like structures.  The heart of the proof of that result was the identification of the fiber of the Koszul map with a polytope formed by a blow-down of a permutahedron.  However these arguments do not apply to the operad $\op{SM}$ -- indeed in this case the fiber need not be a polytope at all.  The main purpose of this article, then, is to prove:

\begin{theorem}\label{thm}  The operad $\op{SM}$ encoding Schwarz modular operads is Koszul.
\end{theorem}

The statement of this theorem presupposes our ability to encode $\op{SM}$ as a quadratic operad.  This can't naively be done with the classical theory of colored operads, and we choose to work with two closely related generalizations: groupoid colored operads and Feynman categories.

To prove Theorem $\ref{thm}$ we first establish

\begin{lemma}\label{dllem} The theory of distributive laws and applications to Koszulity generalizes to the groupoid colored setting.
\end{lemma}

The theory of distributive laws as a way to establish Koszulity of an operad was pioneered by Markl in  \cite{MarklDist} and carried on by Vallette in \cite{Val}.  An expository account can be found in the textbook of Loday-Vallette \cite[Section 8.6]{LV}. We hasten to emphasize that the generalization to the groupoid colored setting is completely straightforward.

Given Lemma $\ref{dllem}$, the heart of the proof of Theorem $\ref{thm}$ is verification of the groupoid colored variant of the Diamond Lemma to show that the Feynman category $\op{SM}$ has operad structure inherited from a distributive law on the pleythism $\op{S}\circ\op{M}$ formed from the Koszul operad of self-gluings $\op{S}$ and the Koszul operad of mergers $\op{M}$.  Once this analysis has been performed, the result follows from the Koszulity of constituents $\op{S}$ and $\op{M}$.  The fact that $\op{S}$ is Koszul follows from our previous work \cite{KFC}, where-as the Koszulity of $\op{M}$ is a straight forward generalization of Koszulity of the commutative operad.

Koszulity results for operads encoding operad-like structures have recently been  developed in a number of cases \cite{WardMP}, \cite{BMO}, \cite{KFC}.  What makes the example of Schwarz modular operads unlike these previous examples is that its encoding operad is not self dual, not even up-to parity reversal.  Indeed $(\op{SM})^!$ is a genuinely new object.  The Feynman transform (aka bar construction) of a Schwarz modular operad is an algebra over $\op{SM}^!$, and the construction, on the one hand, follows from the general theory and yet, on the other hand, is an example of a new type of operad-like structure which seems not to have been studied before.

A parallel situation arises in recent work of Stoeckl \cite{KurtS} who generalizes the theory of Grobner bases to prove Koszulity of operads encoding variants of props and wheeled props.  A similar comment applies to the Koszul duals of these objects, where-in the horizontal composition is commutative up to permutation, hence dual to a Lie-like merging operation.  We feel these new operad like structures are deserving of further study.

Finally, we mention a relationship between Schwarz modular operads and BV algebras.  The totalization of an operad or cyclic operad classically carries a Lie algebra structure.  In the case of a $\mathfrak{K}$-modular operad this Lie structure comes with a compatible differential corresponding to the sum over all self-gluings.  We showed in \cite{KWZ} that if such a $\mathfrak{K}$-modular operad has a compatible merging, this dg Lie structure lifts to a BV algebra.  Crucially this requires the odd or $\mathfrak{K}$ twisted analog of the self-gluings, and we denote the operad encoding such structure by $\op{S}^!\op{M}$.  An immediate consequence of Theorem $\ref{thm}$ tells us that this operad is Koszul, from which we derive:

\begin{corollary}\label{cor}  The totalization of a $\op{S}^!\op{M}_{\infty}$-algebra carries a $BV_\infty$-algebra structure inducing the results of \cite{KWZ} on homology.
\end{corollary}

This paper is organized as follows.  In section $\ref{secrev}$ we give a short review of Feynman categories, groupoid colored operads and the Diamond lemma.  In section 3 we analyze the Feynman category encoding Schwarz modulars operads.  We use this analysis in Section 4 to prove Theorem $\ref{thm}$.  Finally we establish Corollary $\ref{cor}$ in Section 5.

{\small {\it Acknowledgment:}  RK wishes to thank Albert Schwarz for pointing him to his beautiful paper \cite{S} several years ago.  RK gratefully acknowledges support from the Simons Foundation.  BW would also like to gratefully acknowledge support from the Simons Foundation. }

\section{Review of Feynman Categories and Distributive Laws}\label{secrev}  In this section we give a recollection of Feynman categories and groupoid colored operads and adapt the theory of distributive laws to this setting.  This section is largely review, albeit of results that might be a bit spread out.  In addition to our previous work \cite{KW,WardMP, KFC}, we are influenced heavily by the pioneering work of Markl \cite{MarklDist} and the book of Loday and Vallette \cite{LV}.

Throughout this section, we fix a small groupoid $\mathbb{V}$, all of whose automorphism groups $Hom_\mathbb{V}(v,v)$ are finite.  In the subsequent sections of this paper we will specialize to a particular $\mathbb{V}$ whose objects are natural numbers and whose morphisms are symmetric groups.

\subsection{$\mathbb{V}$-trees}\label{treesec}

We take the viewpoint of graphs as consisting of vertices and half-edges. We refer to \cite{KW} for extensive detail on this category of graphs.  Half-edges may be fused in pairs to form edges, the set of which is denoted $Ed(-)$, or may stand alone in which case they are called legs, the set of which is denoted $Leg(-)$.  By an $n$-tree we refer to a simply connected graph with $n+1$ legs, labeled bijectively by the set $\{0\cdc n\}$.  We view an $n$-tree as rooted by declaring the root to be the leg labeled by $0$.  In particular this has us view $n$-trees as directed graphs, directed toward the root, with $n$ ordered input legs, which we call leaves.  

Given a vertex $v$ of a tree $\mathsf{t}$ the arity of $v$ is the number of incoming half edges and the valence is the number of adjacent half edges.  In particular the arity is always one less than the valence.  In this paper we will assume that all vertices have valence $\geq 2$, meaning each vertex has at least one input.  A vertex of arity 1 (valence 2) will be called unary.  A vertex of arity 2 (valence 3) will be called binary.

In an $n$-tree, the union of the set of half-edges with the set of vertices is a partially ordered set.  An element $x$ is greater than an element $y$ if $y$ lies along the unique path connecting $x$ to the root leg.  We use the terminology {\it above} and {\it below} to refer to this partial ordering, with the convention that the root leg is below all other elements.  Given a vertex $v$ of arity $m$ in an $n$-tree $\mathsf{t}$, if we remove $v$ along with the half-edges adjacent to or below $v$, the result is $m$ rooted trees, whose leaves are labeled with sets which together partition $\{1\cdc n\}$.  We call these trees the {\it branches} of $\mathsf{t}$ at $v$.  We order the set of branches at a vertex by declaring branch $b_e$ to be less than branch $b_f$ if the least leaf label of $b_e$ is less than the least leaf label of $b_f$.

A $\mathbb{V}$-coloring of an $n$-tree $\mathsf{t}$ is a function $ Ed(\mathsf{t})\cup Leg(\mathsf{t})\to Ob(\mathbb{V})$. Given a $\mathbb{V}$-colored $n$-tree whose leg $i$ is labeled by object $v_i$, we say the tree is of {\it type} $\vec{v}=(v_1\cdc v_n;v_0)$.  

\subsection{$\mathbb{V}$-corollas}

From $\mathbb{V}$ we define a new small groupoid, called $\mathbb{V}$-corollas, as follows.  The set of objects of $\mathbb{V}$-corollas is
$$
\{ (v_1\cdc v_r; v_0) \ | \ r\geq 1 \text{ and } v_i\in ob(\mathbb{V})   \}.
$$
We will often use the notation $\vec{v}$ to denote an object of $\mathbb{V}$-corollas, and write $|\vec{v}|=r$ for $\vec{v}=(v_1\cdc v_r; v_0)$, with such an object said to be of length $r$.  The symmetric group $S_r$ acts on the set of objects of length $r$ by $$\sigma\vec{v} = \sigma(v_1\cdc v_r; v_0) := (v_{\sigma^{-1}(1)}\cdc v_{\sigma^{-1}(r)};v_0).$$

The set of morphisms $Hom(\vec{v},\vec{w})$ in $\mathbb{V}$-corollas is the empty set unless $\vec{w}=\sigma\vec{v}$ for some $\sigma \in S_r$, in which case
$$
Hom(\vec{v},\sigma\vec{v}) = Aut(v_1)\times ... \times Aut(v_r) \times Aut(v_0)^{op}.
$$

Composition of morphisms $Hom(\vec{v},\sigma\vec{v})\times Hom(\sigma\vec{v},\tau(\sigma\vec{v})) \to Hom(\vec{v},\tau(\sigma\vec{v}))$
is given by 
$$
(\phi_1\cdc\phi_r;\phi_0^{op})\circ (\psi_1\cdc\psi_r;\psi_0^{op}) = (\phi_{\sigma(1)}\psi_1\cdc \phi_{\sigma(r)}\psi_r; \psi^{op}_0\phi_0^{op})
$$

Note $\phi_{\sigma(i)}\psi_i$ are composible in $\mathbb{V}$ since $\sigma(\vec{v})$ has the object $v_i$ in position $\sigma(i)$.

\begin{definition}
	A $\mathbb{V}$-colored sequence is a functor from $\mathbb{V}$-corollas.
\end{definition}
In this paper the target $\op{C}$ of such a functor will be either the category of sets, of vector spaces or of chain complexes, taken over a field of characteristic $0$. 
Unpacking this definition, a $\mathbb{V}$-colored sequence $A$ valued in the category $\op{C}$ is the following data:
\begin{itemize}
	\item An object $A(v_1\cdc v_r;v_0) \in \op{C}$ for each list $(v_1\cdc v_r;v_0)$ of objects in $\mathbb{V}$.
	\item A right action of the group $\times_{i\geq 1} Aut(v_i)$ on $A(v_1\cdc v_r;v_0)$.
	\item A left action of the group $Aut(v_0)$ on $A(v_1\cdc v_r;v_0)$.
	\item A compatible isomorphism $A(v_1\cdc v_r;v_0) \to A(v_{\sigma^{-1}(1)}\cdc v_{\sigma^{-1}(r)};v_0)$ for each $\sigma \in S_r$.
\end{itemize}

Just as uncolored operads are built from symmetric sequences, $\mathbb{V}$-colored operads are built from $\mathbb{V}$-colored sequences.  Indeed, note that if $\mathbb{V}$ were a groupoid with one object and only the identity morphism, then a $\mathbb{V}$-colored sequence would merely be a symmetric sequence.

\subsection{Feynman Category of $\mathbb{V}$-colored trees}  In this section we will consider a monad $\mathbb{T_V}$ on the category of $\mathbb{V}$-colored sequences, built from the notion of a $\mathbb{V}$-colored tree.  The algebras over this monad may equivalently be described as the functors from a cubical Feynman category, from which we access the notions of quadratic and Koszul duality, after \cite{KW}.

To begin, fix a $\mathbb{V}$-colored sequence $A$.  Given a $\mathbb{V}$-colored tree $\mathsf{t}$, any vertex $w\in vert(\mathsf{t})$ determines an unordered list of objects of $\mathbb{V}$ by reading off the colors of the adjacent edges and legs.  Choose any order on the input colors, call them $v_1,\cdc v_r$, and suppose the output of $w$ is labeled by $v_0\in \mathbb{V}$.  Define
$$
A(w) = \left(\coprod_{\sigma\in S_r} A(v_{\sigma(1)}\cdc v_{\sigma(r)};v_0)\right)_{S_r}
$$
with $S_r$ acting diagonally on the sum and the summands, the result being independent of the chosen order on the inputs.

We then define $A(\mathsf{t})\in\op{C}$ of the form
\begin{equation}\label{eq}
A(\mathsf{t}) = \left(\otimes_{w\in vert(\mathsf{t})} A(w)\right)/\sim	
\end{equation}
where $\sim$ denotes the coinvariants by the action of $Aut(v_e)$ along any edge $e$ of color  $v_e\in \mathbb{V}$.  In particular if $e$ is an edge adjacent to vertices $x$ below $y$ then $Aut(v_e)$ acts on $A(x)$ on the right and $A(y)$ on the left and we take the balanced tensor product over all edges of $\mathsf{t}$.

Define $\mathbb{T_V}(A)(\vec{v})$ to be the coproduct of the $A(\mathsf{t})$ over all isomorphism classes of $\mathbb{V}$-colored trees $\mathsf{t}$ of type $\vec{v}$.  This allows us to view $\mathbb{T_V}\circ\mathbb{T_V}(A)$ as nested $\mathbb{V}$-colored trees, i.e.\ as a $\mathbb{V}$-colored tree each of whose vertices is in turn labeled with a $\mathbb{V}$-colored tree of matching type.  Upon forgetting the nesting, such a nested $\mathbb{V}$-colored tree determines a $\mathbb{V}$-colored tree, which determines a natural transformation $\mathbb{T_V}\circ\mathbb{T_V}\Rightarrow \mathbb{T_V}$ which induces the structure of a monad on $\mathbb{T_V}$.  

There is a symmetric monoidal category, $\hat{\mathbb{T}}_{\mathbb{V}}$, for which $Fun_\tensor(\hat{\mathbb{T}}_{\mathbb{V}},\op{C})\cong \mathbb{T_V}$-algebras in $\op{C}$.  The objects of $\hat{\mathbb{T}}_{\mathbb{V}}$ are monoidal products of lists of the form $(v_1\cdc v_n;v_0)$.  The morphisms in $\hat{\mathbb{T}}_{\mathbb{V}}$ are of two types: automorphisms and monoidal products of $\mathbb{V}$-colored trees; the source of such a morphism is indicated by cutting the edges of the tree and the target by contracting the edges.  The category $\hat{\mathbb{T}}_{\mathbb{V}}$ is an example of a Feynman category \cite{KW}.  Moreover, this example is a cubical Feynman category, hence Koszul \cite{KFC}.  This result gives us access to Koszul duality in the category of $\mathbb{T}_{\mathbb{V}}$-algebras.

We define a $\mathbb{V}$-colored operad to be an algebra over $\mathbb{T_V}$.
This means in particular that for every $\mathbb{V}$-colored tree $\mathsf{t}$ of type $\vec{v}$ there is a composition map
$
A(\mathsf{t}) \to A(\vec{v}).
$  
These compositions are in turn generated by single edged compositions
\begin{equation}\label{circi}
A(v_1\cdc v_r; v_0)\tensor A(w_1\cdc w_s; v_i) \stackrel{\circ_i}\to A(v_1\cdc v_{i-1}, w_1\cdc w_s,v_{i+1}\cdc v_r; v_0 ),
\end{equation}
which are equivariant with respect to the balanced $Aut(v_i)$ action.

The following example is helpful in keeping track of the axiomatics.

\begin{example}\label{algebraex}
	Let $X$ be a $\mathbb{V}$-module, i.e.\ a functor $\mathbb{V}\to\op{C}$.  There is an associated $\mathbb{V}$-colored sequence $End_X$ given by
	$$
	End_{X}(\vec{v}) = Hom_\op{C}(X(v_1)\tensor\dots\tensor X(v_n), X(v_0))
	$$
	and with the obvious automorphism and symmetric group actions.  This $\mathbb{V}$-colored sequence has the structure of a $\mathbb{V}$-colored operad by composition of functions.  By definition, an algebra over a $\mathbb{V}$-colored operad $\op{P}$ is a morphism $\op{P}\to End_X$.
\end{example}

Strictly speaking, $\mathbb{T_V}$ is the monad for non-unital $\mathbb{V}$-colored operads.  A related notion of unital $\mathbb{V}$-colored operads can be defined as the algebras over a slight modification of the monad $\mathbb{T_V}$ by adjoining the $\mathbb{V}$-colored ``trees'' with no vertices.  Algebras over the unital version of this monad may be equivalently defined as monoids for a monoidal product of $\mathbb{V}$-colored sequences.  This monoidal products was first given by Petersen \cite{Pet}, see also Theorem 2.24 of \cite{WardMP}.  Given two $\mathbb{V}$-colored sequences $A$ and $B$, we denote this monoidal product by $A\circ B$.  It is defined as a coproduct over trees having two levels, the lower labeled by $A$ and the upper labeled by $B$, subject to the identifications of moving an automorphism along an edge as in Equation $\ref{eq}$.

\subsection{Quadratic $\mathbb{V}$-operads} 
By the general theory for algebras over a monad, the forgetful functor from $\mathbb{T_V}$-algebras to $\mathbb{V}$-colored sequences has a left adjoint $F$ with $F(A) = \mathbb{T_V}(A)$, so we adopt the notation $F(-)$ in what follows.  We remark that in general, the coproduct $F(A)=\mathbb{T_V}(A)$ could be infinite, but in the examples within this paper this is not the case i.e.\ our examples are reduced in the sense of \cite[Definition 2.35]{WardMP}.

Let's now specialize to the category $\op{C}=Vect$.  Assign a grading, called the weight, to a vector space $F(A)(\vec{v})$ by declaring $A(\mathsf{t})$ to be weight $|vert(\mathsf{t})|$.  Let us denote the weight $i$ summand by $F(A)^i$.  Recall we are working with the non-unital variant of groupoid colored operads, and so the weight is always positive.

Given a collection of subsets $B(\vec{v})\subset F(A)(\vec{v})$ we define the ideal generated by $B$, denoted $\langle B \rangle$, to be the smallest $\mathbb{V}$-colored sequence containing $B$ which is closed under the structure maps of the form $b\circ_i a$ and $a\circ_j b$ for $b\in B$.  A $\mathbb{V}$-colored operad $\op{P}$ is quadratic if there exists $\mathbb{V}$-colored sequences $A$ and $R$ with $R\subset F(A)^2$, such that $\op{P}\cong F(A)/\langle R \rangle$.  The $\mathbb{V}$-colored sequences $A$ and $R$ are called the generators and relations of $\op{P}$ respectively.

Conversely, given a $\mathbb{V}$-colored sequence $A$ and a $\mathbb{V}$-colored sequence $R\subset F(A)^{(2)}$, we denote the associated quadratic operad simply by $(A,R)$.  By abuse of terminology, given any family of sets $R^\prime(\vec{v}) \subset R(\vec{v})$ such that $\langle R^\prime\rangle=\langle R \rangle$, we may write $(A,R)=(A,R^\prime)$, given that they generate the same quadratic operad.  In particular $R^\prime \subset R \subset \langle R \rangle$ may all be referred to as ``the'' relations.

Given a quadratic $\mathbb{V}$-colored operad, we define its quadratic dual in analogy with the classical theory.  Specifically, in the case that $A$ and $R$ are finite dimensional for each $\vec{v}$, we may define the quadratic dual operad to be $(A,R)^! := (\Sigma A^\ast, R^\perp)$ where $\Sigma$ denotes a shift in degree, $A^\ast$ denotes the linear dual and $R^\perp\subset F(\Sigma A^\ast)$ denotes those linear functionals which vanish on $R$.

Since the category of $\mathbb{T_V}$-algebras is the category of representations of a cubical, hence Koszul, Feynman category, it makes sense to apply the bar/cobar construction, see \cite[Section 2.7]{WardMP}.  In analogy with the classical setting, we define a quadratic $\mathbb{V}$-colored operad to be Koszul if the canonical map between the linear dual of the bar construction and the quadratic dual operad is a quasi-isomorphism.  

\subsection{The generalized Diamond lemma.}

In \cite{MarklDist}, Markl generalized the notion of distributive laws from associative algebras to operads, and proved that an operad defined by a distributive law is Koszul if its constituent pieces are as well.  These techniques were refined by Vallette \cite{Val} and given an expository account by Loday and Vallette in \cite[Section 8.6]{LV}.  The goal of this section is to observe that these now standard Koszulity results lift from uncolored operads to groupoid colored operads.

Let $V$ and $W$ be $\mathbb{V}$-colored sequences.  The weight $2$ summand of the free operad $F(V\oplus W)^{(2)}\subset F(V\oplus W)$ may be decomposed into a direct sum of four $\mathbb{V}$-colored sequences,
$$
F(V\oplus W)^{(2)} = F(V)^{(2)} \oplus F(W)^{(2)} \oplus V\circ_{(1)}W\oplus W\circ_{(1)}V,
$$
where $V\circ_{(1)}W$ denotes those 2-vertex trees whose lower vertex is labeled by $V$ and upper vertex is labeled by $W$ and vice-versa for $W\circ_{(1)}V$.  A natural transformation $\lambda\colon W\circ_{(1)}V \Rightarrow V\circ_{(1)}W$ will be called a rewriting rule.   Given such a rewriting rule $\lambda$, we define the $\mathbb{V}$-colored sequence 
$$D_\lambda:=\{ \mathsf{t}-\lambda(\mathsf{t}) \colon \mathsf{t} \in W\circ_{(1)} V \} \subset F(V\oplus W)^{(2)}.$$

Now suppose $\op{A}=(V,R)$ and $\op{B}=(W,S)$ are quadratic $\mathbb{V}$-colored operads.  In the presence of a rewriting rule $\lambda\colon W\circ_{(1)}V \Rightarrow V\circ_{(1)}W$ we define the quadratic operad $$\op{A}\vee_\lambda \op{B}:=(V\oplus W, R\oplus D_\lambda\oplus S). $$

There are morphisms of $\mathbb{V}$-colored sequences $\op{A}\circ\op{B}\hookrightarrow F(V\oplus W) \twoheadrightarrow \op{A}\vee_\lambda \op{B}$.  Call the composite $p$.  Each $p(\vec{v})$ is surjective, since a tree labeled by $V$ and $W$ is equivalent, through repeated application of the relations in $D_\lambda$, to a tree for which the vertices labeled by $V$ are below the vertices labeled by $W$.

We now state the groupoid colored version of the diamond lemma for distributive laws.  Although technically a generalization of the classical operad case,  the proof follows verbatim from that of eg \cite[Theorem 8.6.5]{LV}.

\begin{lemma} (Diamond Lemma) \label{diamondlem}
	Let $\op{A}=(V,R)$ and $\op{B}=(W,S)$ be Koszul operads and let $\lambda\colon W\circ_{(1)}V \Rightarrow V\circ_{(1)}W$ be a rewriting rule.  If $p$ restricted to $\op{A}\circ\op{B}^{(3)}(\vec{v})$ is injective for each $\vec{v}$, then $p$ is an isomorphism of $\mathbb{V}$-colored sequences and the operad $\op{A}\circ\op{B}\cong \op{A}\vee_\lambda \op{B}$ is Koszul.
\end{lemma}

\section{Schwarz Modular Operads}

In this section we define a groupoid colored operad $\op{SM}$ whose algebras are Schwarz modular operads.  We fix the groupoid $\mathbb{V}$ for the remainder of this article as follows:
\begin{equation}\label{V}
ob(\mathbb{V}) = \mathbb{N} \ \ \text{ and } \ \ Hom_\mathbb{V}(n,m) =\begin{cases} S_n \ & \text{ if } m=n \\ \emptyset & \text{ if } m \neq n \end{cases}
\end{equation}
where $S_n$ denotes the symmetric group. 

First we recall the definition of Schwarz modular operads. 
 This definition makes sense in any symmetric monoidal category $\op{C}$.  However, subsequently we will focus on the cases of $\op{C}=\mathsf{Sets}$ and $\op{C}=\mathcal{V}ect$.  We often use the mathsf and mathcal fonts to distinguish objects in these two categories, e.g.\ $\text{span}(\mathsf{X}) = \op{X}$.

\begin{definition}\label{smodef} \cite{S} A Schwarz modular operad (abbreviated SMO) in a symmetric monoidal category $(\op{C},\tensor)$ is a symmetric sequence $A$ in $\op{C}$ along with morphisms
	\begin{itemize}
		\item $\ast_{n,m}\colon A(n)\tensor A(m) \to A(n+m)$, called mergers,
		\item $\xi_{n-1,n} \colon A(n)\to A(n-2)$, called self-gluings,
	\end{itemize} 
	which satisfy the following seven axioms.  Here we write $\cdot$ for composition of functions.
	\begin{enumerate}
		\item[(S1)] Equivariance of self-gluings I:  $\sigma \cdot\xi_{n-1,n} = \xi_{n-1,n}\cdot \sigma$ for $\sigma \in S_{n-2}\subset S_n$.
		\item[(S2)] Equivariance of self-gluings II: $\xi_{n-1,n} = \xi_{n-1,n}\cdot (n-1 \ n)$, where $(n-1 \ n)\in S_{n}$ denotes the transposition exchanging $n-1$ and $n$.	
		\item[(S3)] Equivariance of mergers: $(\sigma,\tau)\cdot\ast_{n,m} =  \ast_{n,m}\cdot(\sigma,\tau)$ for $\sigma \in S_n,\tau\in S_m$.
		\item[(S4)] Symmetry of mergers: $\sigma_{n,m} \cdot\ast_{n,m} =  \ast_{m,n} \cdot s_{A(n),A(m)}$ where $s$ denotes the swap map in the symmetric monoidal category $\op{C}$ and $\sigma_{n,m}$ is the $(n,m)$-shuffle which adds $m$ to $1,...,n$ and subtracts $n$ from $n+1,..., n+m$.
		\item[(R1)] Commutativity of self-gluings: $\xi_{n-3,n-2}\cdot \xi_{n-1,n} = \xi_{n-3,n-2} \cdot \xi_{n-1,n} \cdot \sigma$ for $\sigma=(n-1 \ n-3)(n-2 \ n)$.
		\item[(R2)] Associativity of mergers: $\ast_{n+m,l}\cdot(\ast_{n,m}\tensor id_{A(l)}) = \ast_{n,m+l}\cdot (id_{A(n)}\tensor\ast_{m,l})$
		\item[(R3)]  Distributivity of mergers over self-gluings $\xi_{n+m-1,n+m}\cdot \ast_{n,m} = \ast_{n,m-2}\cdot (id_{A(n)}\tensor\xi_{m-1,m})$.
	\end{enumerate}		
\end{definition}
This definition is taken almost verbatim from \cite{S}, except the notation is different and the axioms have been reordered.

\subsection{Self-gluings and reindexing}  Given a SMO, and integers $1\leq i<j \leq n$ we define $\xi_{i,j}:= \xi_{n-1,n}\cdot \rho_{i,j}$, where $\rho_{i,j}\in S_n$ is the unique permutation satisfying $\rho_{i,j}(i)=n-1, \rho_{i,j}(j)=n$ and $a<b \Rightarrow \rho_{i,j}(a)<\rho_{i,j}(b)$ for $a,b \in \{1\cdc n\}\setminus \{i,j\}$. 
Observe that, unlike in Definition $\ref{smodef}$, the expressions $\xi_{i,j}$ and $\rho_{i,j}$ assume $n$ has been fixed, and we write $\xi_{i,j}^n$ (resp.\ $\rho_{i,j}^n$) if we need to indicate the source.  The operations $\xi_{i,j}$ will also be referred to as self-gluings.

Below we will need to work carefully with the operations $\xi_{i,j}$ so let us give an explicit description of the permutation $\rho_{i,j}$. For this we make the following definition.  Let $a,b,c$ be distinct natural numbers with $a<b$.  We define $\epsilon_{a,b}(c)$ to be:
\begin{equation*}
	\epsilon_{a,b}(c) = \begin{cases}
		0 & \text{ if } c<a \\
		1 & \text{ if } a<c<b \\
		2 & \text{ if } b<c \\
	\end{cases}
\end{equation*}

With this definition we have
\begin{lemma} \label{rhocor}
The permutation $\rho_{i,j}\in S_n$ is given by
$$\rho_{i,j}(k) = 
\begin{cases}
n-1 & \text{ if } k = i \\ 
n & \text{ if } k = j \\ 
k - \epsilon_{i,j}(k) & \text{ else.} \\ 
\end{cases}
$$
\end{lemma}
\begin{proof}  The first two rows are a restatement of the definition of $\rho_{i,j}$.  For the third line it suffices to check that given $i,j$, the function $k - \epsilon_{i,j}(k)$ is increasing on the domain $\{1\cdc n\}\setminus\{i,j\}$.\end{proof}

\subsection{Two self gluings.}\label{2sec}  We record some terminology related to the reindexing of pairs of self-gluings for future use.

Fix $n\geq 4$.  An ordered pair of self gluings is a 4-tuple $(a,b,c,d)$ such that $1\leq a < b \leq n$ and $1\leq c < d \leq n-2$. Let $\mathsf{O}=\mathsf{O}_n$ be the set of ordered pairs of self-gluings.

An unordered pair of self-gluings is an orbit of the set of 4-tuples $(i,j,k,l)$ of distinct natural numbers under the action of the centralizer of $(12)(34)\in S_4$.  Let $\mathsf{U}=\mathsf{U}_n$ be the set of unordered pairs of self-gluings having entries $\leq n$.

There is a map $\phi\colon\mathsf{O}\to\mathsf{U}$ given by the formula
\begin{equation}\label{phidef}
\phi(i,j,\rho_{i,j}(k),\rho_{i,j}(l)) = [(i,j,k,l)]
\end{equation}
note such an $i,j,k,l$ are distinct since their images under $\rho_{i,j}$ are distinct.  The map $\phi$ is surjective and the fiber $\phi^{-1}(u)$ has size two for each $u\in\mathsf{U}$.  To see this, note that there are two representatives $(i,j,k,l)$ of a given $u$ for which $i<j$ and $k<l$, leading to a fiber of at least two, but on the other hand one easily counts the orders of these sets to confirm $|\mathsf{O}|/|\mathsf{U}|=2$.

Define an involution $\iota\colon \mathsf{O}\to\mathsf{O}$ by switching the elements in the fiber of $\phi$.  In particular $\iota$ is characterized by the properties $\iota(a,b,c,d) \neq (a,b,c,d)$ and $\phi(\iota(a,b,c,d)) = \phi((a,b,c,d))$.  It is given by the formula
\begin{equation}\label{iota}
\iota(i,j,\rho_{i,j}(k),\rho_{i,j}(l)) = (k,l,\rho_{k,l}(i),\rho_{k,l}(j)).
\end{equation}
We write $\iota_n$ in place of $\iota$ if we need to specify the index of the source and target.  The following lemma is straightforward.

\begin{lemma}\label{plusn}
	Fix a natural number $n$ and define $p_n\colon \mathsf{O}_m \to\mathsf{O}_{n+m}$ by the formula $p_n(i,j,k,l)=(i+n,j+n,k+n,l+n)$.  Then $p_n\circ\iota_m = \iota_{n+m}\circ p_n$.
\end{lemma}

\subsection{Quadratic Presentation of $\op{SM}$.}

The Feynman category encoding Schwarz modular operads was introduced in \cite{KW} and considered in greater detail in \cite{CB22}.  Specializing to the case $\op{C}=\op{V}ect$ or $dg\op{V}ect$, Schwarz' original definition immediately lends itself to a quadratic presentation of this Feynman category. In this subsection, we analyze this presentation in detail.

\subsubsection{Generators of $\op{SM}$.}
Recall (Equation $\ref{V}$) we have fixed $\mathbb{V}$ to be the groupoid whose objects are natural numbers and whose automorphisms are symmetric groups. In particular a $\mathbb{V}$-module is a symmetric sequence, i.e.\ a sequence of objects $A(n)$ along with an action of the symmetric group $S_n$ on $A(n)$.   We shall define a quadratic $\mathbb{V}$-colored operad $\op{SM}:=F(\op{E})/\langle \op{R} \rangle$ with the property that a symmetric sequence is a SMO in Vect if and only if it is an $\op{SM}$-algebra in Vect.  

For this, we first work in the category of sets and define a $\mathbb{V}$-colored sequence of generators, call it $\mathsf{E}$, in the category of sets:
\begin{equation}\label{generators}
\mathsf{E}(\vec{v}) = \begin{cases}
	\xi_{n-1,n}\times (S_{n}/\langle (n-1,n) \rangle)     & \text{ if } \vec{v}=(n;n-2) \\
	S_{n+m}\times \ast_{n,m} & \text{ if } \vec{v}=(n,m;n+m) \\ 
	\emptyset	 & \text{ else } 
\end{cases}
\end{equation}
The symmetric group actions are as follows.  The group $S_n$ acts on the right of $\mathsf{E}(n;n-2)$ by multiplication on cosets.  The group $S_{n-2}$ acts on the left of $\mathsf{E}(n;n-2)$ by multiplication on cosets.  We regard $S_n\times S_m\subset S_{n+m}$ as usual ($S_n$ permutes the first $n$ and $S_m$ permutes the last $m$ objects).  Then $S_n\times S_m$ acts on the $S_{n+m}$ factor of $\mathsf{E}(n,m;n+m)$ by right multiplication and $S_{n+m}$ acts by left multiplication.  Finally, the generator of $S_2$ corresponds to the unique $S_{n+m}$-equivariant map $\mathsf{E}(n,m;n+m) \to \mathsf{E}(m,n;n+m)$ sending $\ast_{n,m}$ to $\sigma_{n,m} \cdot \ast_{m,n}$, where $\sigma_{n,m}$ is as in Definition $\ref{smodef}$.

As above we then define $\xi_{i,j}:= \xi_{n-1,n}[\rho_{i,j}]$ where $[\rho_{i,j}]$ denotes the coset of $\rho_{i,j}$ in the set $S_n/\langle (n-1 \ n) \rangle.$

\begin{lemma}\label{ijlem}
	$\mathsf{E}(n,n-2) = \{\sigma\xi_{i,j} \ | \ \sigma \in S_{n-2} \text{ and } 1\leq i < j \leq n \}$.
\end{lemma}
\begin{proof}
	Equal here means that we consider the elements on the right as elements on the left, in the sense that  $\sigma\xi_{i,j} = \sigma\xi_{n-1 , n} [\rho_{i,j}] = \xi_{n-1 , n} [\sigma\rho_{i,j}]$. Thus it remains to show that the set of permutations $\sigma\rho_{i,j}$ over all $\sigma\in S_{n-2}$ and all $1\leq i<j \leq n$ forms a set of representatives for the cosets $S_n/\langle (n-1 \ n) \rangle$. 
	
	For this, observe that $\sigma \rho_{i,j}\rho_{k,l}^{-1}\tau^{-1}$ can not exchange $n-1$ and $n$, it fixes $n-1$ iff $i=k$ and fixes $n$ iff $j=l$.  Therefore this permutation is in $\langle (n-1 \ n) \rangle$ only if $\rho_{i,j}= \rho_{k,l}$, which would in turn imply that $\sigma\tau^{-1}\in S_{n-2}\subset S_n$ is the identity.  Hence different expressions of the form $\sigma \rho_{i,j}$ represent different cosets.  Since the number of cosets and the number of such expressions coincide (namely $\frac{n!}{2}=(n-2)!{n \choose 2}$) the claim follows.
\end{proof}

Finally, we define $\op{E}:=\text{span}\{\mathsf{E}\}$, with the inherited symmetric group actions.

\subsubsection{Description of $F(\mathsf{E})$}  By definition, the elements of each set $F(\mathsf{E})(\vec{v})$ are equivalence classes of $\vec{v}$-labeled trees.  However, in this example it is possible to identify a distinguished representative of each class, and in turn describe these sets as a particular class of labeled trees, as we now demonstrate.

Recall that an $\mathsf{E}$-labeled tree refers to a rooted tree, colored by $\mathbb{V}$ along with an element $\mathsf{E}(w)$ for each vertex $w$ of the tree.  In this example, a vertex in an $\mathsf{E}$-labeled tree necessarily has arity 1 or 2.
  
  For a vertex $w$ of arity $1$ we have $E(w) = E(n,n-2)$ for some $n$.  For a vertex $u$ of arity 2 we have $E(u) \cong (E(n_1,n_2;n_1+n_2)\coprod E(n_2,n_1;n_1+n_2))_{S_2}$.  In particular, a bivalent vertex $u$ joins two branches of the tree whose roots are colored by the integers $n_1$ and $n_2$ respectively.  The choice of a representative of the coinvariants $E(u)$ specifies an order of the branches.  Recall (subsection $\ref{treesec}$), we say the branch corresponding to $n_1$ is less than the branch corresponding to $n_2$ if the least leaf label on branch $n_1$ is less than the least leaf label on branch $n_2$.
  
  We call an $\mathsf{E}$-labeled tree ``pure'' if each univalent label is of the form $\xi_{i,j}$ and each bivalent label is of the form $[\ast_{n,m}]$, where the branch corresponding to $n$ is less than the branch corresponding to $m$.  A key feature of a pure tree is that no non-trivial automorphisms appear on the vertex labels. Note that a pure tree can be depicted by a planar, leaf-labeled rooted tree with the property that the least leaf label above any trivalent vertex appears on the left, along with $\xi_{i,j}$ labels of all bivalent vertices.  Define $\mathsf{T}_{\ast,\xi}(\vec{v})$ to be the set of isomorphisms classes of pure $\mathsf{E}$ labeled trees of type $\vec{v}$.  

\begin{lemma}\label{purelem}  For each $\vec{v}=(n_1\cdc n_r; n_0)$ there is a canonical bijective correspondence $F(\mathsf{E})(\vec{v}) \cong S_{n_0}\times\mathsf{T}_{\ast,\xi}(\vec{v}) $ which is compatible with the left $S_{n_0}$-action.
\end{lemma}
\begin{proof}
	The bijection is given as follows.  On the one hand, given a pure $\vec{v}$-tree and a permutation $\sigma\in S_{n_0}$, we get an element in $F(\mathsf{E})(\vec{v})$ by starting with the given tree and acting on the left of the root label by $\sigma$.
	
	On the other hand, given an element in $F(\mathsf{E})(\vec{v})$, choose a representative.  Such a representative is comprised of a tree $\mathsf{t}$ having only uni- and bi- valent vertices, along with labels of the vertices of the from $\sigma\xi_{i,j}$ (univalent case)  and $\sigma\ast_{n,m}$ (bivalent case).  If a given vertex $w$ is not the root vertex, then there is a unique internal edge below and adjacent to it.  Call this edge $e$, and let $u$ be the vertex below and adjacent to $e$. In this case we can form a new representative of the same element in $F(\mathsf{E})(\vec{v})$ by moving the permutation $\sigma$ from the left of the label of $w$ to the right of the label of $u$.  We may then rewrite the label of $u$ in the above form.  Let us call this process of choosing a new representative ``pushing $\sigma$ down along $e$.''.
	
	Now choose an order on the edges of $\mathsf{t}$ such that if $e$ is above $e^\prime$ then $e>e^\prime.$  Push the permutations labeling the vertices down along the edges in the given order.    Different choices of edge orders which satisfy the above condition will give the same result because when encountering a bivalent vertex, the two possible permutations to be moved from right to left commute (one permutes the first $n$ letters the other permutes the last $m$).  The result of this process, then, is a pure labeled tree, along with a permutation labeling the root, hence an element of $S_{n_0}\times\mathsf{T}_{\ast,\xi}(\vec{v})$.

The facts that this correspondence is equivariant under left multiplication and bijective are immediate. \end{proof}

An algebra $A$ over $F(\mathsf{E})$ (in $\mathsf{Sets}$) or over $F(\mathcal{E})$ (in $\op{V}ect$) has an operation $\xi_{n-1,n}\colon A(n)\to A(n-2)$ which we call a self-gluing and an operation $A(n)\tensor A(m)\to A(n+m)$ which we call a merger.   The equivariance axioms for an algebra over an operad ensure $(S1)-(S4)$ are satisfied.  For axioms (R1)-(R3) we must impose additional relations on the free operad.

\subsubsection{The relations of $\op{SM}$}
We define $\op{R}\subset F(\op{E})$ to be the smallest $\mathbb{V}$-colored subsequence such that:
\begin{itemize}
			\item Each $\op{R}(n,n-4)$ contains $\xi_{n-3,n-2}\cdot \xi_{n-1,n} - \xi_{n-3,n-2} \cdot \xi_{n-1,n} \cdot \sigma$ for $\sigma=(n-1 \ n-3)(n-2 \ n)$.
		\item Each $\op{R}(n,m; n+m-2)$ contains $\xi_{n+m-1,n+m}\cdot \ast_{n,m} - \ast_{n,m-2}\cdot (id\tensor \xi_{m-1,m})$, and
	
	\item Each $\op{R}(n,m,l; n+m+l)$ contains $\ast_{n+m,l}\circ_1\ast_{n,m}-\ast_{n,m+l}\circ_2 \ast_{m,l}$,

\end{itemize}
The notation $\cdot$ here means composition of functions, and in particular takes the place of the usual notation $\circ_1$ for 1-ary functions.

Finally we define the associated quadratic operad $\op{SM}:= F(\op{E})/\langle \op{R} \rangle$.  This description of $\op{R}$ is defined to model axioms (R1)-(R3) of Definition $\ref{smodef}$ and hence:


\begin{proposition}  A Schwarz modular operad (in Vect) is precisely an $\op{SM}$-algebra. 
\end{proposition}

The $\mathbb{V}$-colored sequence $\op{R}$ is generated by the relations given above and the symmetric group actions.  Below we will have need to work with a basis for $\op{R}$, and we devote the remainder of this section to describing such a basis.

First we record the following fact about the permutations $\rho$.
\begin{lemma}\label{rholem} 
	Fix distinct $i,j,k,l$ with $1\leq i<j\leq n$ and $1 \leq k <l \leq n$.  Define a new 4-tuple $i^\prime, j^\prime, k^\prime, l^\prime:= \rho_{k,l}(i), \rho_{k,l}(j), \rho_{i,j}(k), \rho_{i,j}(l)$.  Then,
	$$ \rho^{n-2}_{k^\prime,l^\prime}\rho^n_{i,j} = (n-1, n-3)(n-2, n) \rho^{n-2}_{i^\prime,j^\prime}\rho^n_{k,l}.$$ 
\end{lemma}
\begin{proof}  We apply Lemma $\ref{rhocor}$.  The left hand side sends $i$ to $n-1$.  The right hand side sends $i$ to $i^\prime$ and then to $n-3$ and then to $n-1$.  Similarly for $j,k,l$.  For other inputs, we appeal to the order preserving property.
\end{proof}

This definition captures the reindexing which is required to switch the order of the self-gluings, and will be used to prove the first statement of the following Proposition.  We continue use of the notation $i^\prime$ etc.\ as used in Lemma $\ref{rholem}$.

\begin{proposition} \label{basis} The relations $\op{R}$ admit the following bases:
\begin{enumerate}
	\item 	The vector space $\op{R}(n,n-4)$ has  basis
		$$
		\left\{\sigma \cdot ( \xi_{k^\prime l^\prime}\cdot\xi_{ij} - \xi_{i^\prime j^\prime}\cdot\xi_{k l}  ) \right\} $$
		over all $\sigma \in S_{n-4}$ and all $[(i,j,k,l)]\in \mathsf{U}_n$, with convention that we choose a representative of $[(i,j,k,l)]$ for which $j>l$.
		
\item The vector space  $\op{R}(n,m;n+m-2)$ has basis given by the union
$$\{ \sigma \cdot (\xi_{i,j}\cdot \ast_{n,m} - \ast_{n-2,m}\cdot (\xi_{i,j}\tensor id) )  \} \ds\cup \left\{\sigma \cdot (\xi_{k+n,l+n}\cdot \ast_{n,m} - \ast_{n,m-2}\cdot (id \tensor \xi_{k,l}))   \right\},$$ 	
taken over all $\sigma\in S_{n+m-2}$, $1\leq i<j\leq n$ and $1\leq k<l\leq m$.	

\item The vector space $\op{R}(n,m,l;n+m+l)$ has basis given by the union of
$$
\left\{ \sigma \cdot (\ast_{n+m,l}\circ_1\ast_{n,m}-\ast_{n,m+l}\circ_2\ast_{m,l}) \ | \ \sigma \in S_{n+m+l}\right\}
$$
and 
$$(23)\left\{ \sigma \cdot (\ast_{n+l,m}\circ_1\ast_{n,l}-\ast_{n,l+m}\circ_2\ast_{l,m}) \ | \ \sigma \in S_{n+m+l}\right\}$$
where $(23)$ denotes the symmetry $ \op{R}(n,l,m;n+l+m) \stackrel{(23)}\longrightarrow \op{R}(n,m,l;n+m+l)$ (i.e.\ leaf relabeling). 
\end{enumerate}

\end{proposition}

\begin{proof}
	Let's use the notation $B(\vec{v})$ to refer to these bases, viewed as subsets of the vector spaces $\op{R}(\vec{v})$.  In particular $B(\vec{v})$ is empty unless $\vec{v}$ is in one of the 3 families indicated above.

{\bf Statement $(1)$:} The vector space $\op{R}(n,n-4)$ is spanned by the $S_n$ orbit of the vector $$\xi_{n-3,n-2}\cdot \xi_{n-1,n} - \xi_{n-3,n-2} \cdot \xi_{n-1,n} \cdot (n-1, n-3)(n-2, n).$$  This follows from the fact that, in this case, the left action of $\sigma \in S_{n-4}$ coincides with the right action of $S_n$ upon restriction to $S_{n-4}\subset S_n$.  This vector is fixed by the subgroup generated by the three permutations $(n-1, n-3)(n-2, n)$, $(n-1 \ n)$ and $(n-3 \ n-2)$ in $S_n$.  This subgroup has order 8, so we can say $\op{R}(n,n-4)$ has dimension at most $n!/8$.  On the other hand, the set $B(n,n-4)$ contains $3(n-4)!{n\choose 4} = n!/8$ vectors.  Thus to show the set $B(n,n-4)$ is a basis for $\op{R}(n,n-4)$ it is sufficient to show it is both linearly independent and contained within $\op{R}(n,n-4)$.

We first show $B(n,n-4)\subset \op{R}(n,n-4)$.  By closure under $S_{n-4}$, it's enough to show that each $\xi_{k^\prime,l^\prime}\xi_{i,j} - \xi_{i^\prime,j^\prime}\xi_{k,l} \in \op{R}(n,n-4)$.  For this
\begin{align*}
\xi_{k^\prime,l^\prime}\xi_{i,j} - \xi_{i^\prime,j^\prime}\xi_{k,l} 
& = \xi_{n-3,n-2}\xi_{n-1,n}\rho^{n-2}_{k^\prime,l^\prime}\rho^n_{i,j} - \xi_{n-3,n-2}\xi_{n-1,n}\rho^{n-2}_{i^\prime,j^\prime}\rho^n_{k,l} \\
& = \xi_{n-3,n-2}\xi_{n-1,n}(n-3 \ n-1)(n-2 \ n)\rho_{i^\prime,j^\prime}\rho_{k,l} - \xi_{n-3,n-2}\xi_{n-1,n}\rho_{i^\prime,j^\prime}\rho_{k,l} \\ 
& = (\xi_{n-3,n-2}\xi_{n-1,n}(n-3 \ n-1)(n-2 \ n)- \xi_{n-3,n-2}\xi_{n-1,n})\rho_{i^\prime,j^\prime}\rho_{k,l} 
\end{align*}
The penultimate step follows from Lemma $\ref{rholem}$.  The final expression is the permutation of an element in $\op{R}(n,n-4)$ hence is in $\op{R}(n,n-4)$.  Thus $B(n,n-4)\subset \op{R}(n,n-4)$.

 Finally, to see that the set $B(n,n-4)$ is linearly independent, we first observe (after Lemma $\ref{purelem}$), that since $F(\mathsf{E})(n,n-4) \cong S_{n-4}\times \mathsf{T}_{\ast, \xi}(n,n-4)$, it would be enough to show that the set $
	\left\{\xi_{k^\prime l^\prime}\cdot\xi_{ij} - \xi_{i^\prime j^\prime}\cdot\xi_{k l}   \right\} $ is linearly independent in $\op{T}_{\ast, \xi}(n,n-4):= \text{span}\{\mathsf{T}_{\ast, \xi}(n,n-4)\}$.  But here, each tree only appears in a single relation from which linear independence follows.  To see this note that the indices of any two trees appearing in such an expression are related by the involution $\iota$ of subsection $\ref{2sec}$, namely $\iota(i,j,k^\prime,l^\prime) = (k,l,i^\prime, j^\prime)$, and hence coincide after passage to $\mathsf{U}_n$.

{\bf Statement $(2)$: }
The relation defining $\op{R}(n,m;n+m-2)$ is in the putative basis, so if the span of the basis, call it $B=B(n,m;n+m-2)$, is closed under symmetric group actions, it will imply $\op{R}(n,m;n+m-2) \subset span\{B\}.$

Now $B$ is closed under the left $S_{n+m-2}$ action by definition and the $S_2$ action simply switches the two subsets defining $B$.  So it suffices to check that each subset is closed under the right $S_n\times S_m$ action, and for this it would be enough to rewrite an expression of the form
$$(\xi_{i,j} \cdot \ast_{n,m} - \ast_{n-2,m}\cdot (\xi_{i,j}\tensor id) )\cdot (\tau \times id)$$
as an element of $B$.  For this, one can directly verify that if $\rho^n_{i,j}\tau = \sigma\rho^n_{k,l}$
then
$\rho^{n+m-2}_{i,j}(\tau,id) = (\sigma, id)\rho^{n+m-2}_{k,l}$, by direct inspection.

Conversely, we show each element of $B$ is in $\op{R}(n,m;n+m-2)$ and conclude $span\{B\}$ is in $\op{R}(n,m;n+m-2)$.  Since $\op{R}$ is closed under the left symmetric group action it is enough to consider basis elements of the form 
$\xi_{k+n,l+n} \cdot \ast_{n,m} - \ast_{n-2,m}\cdot (id\tensor\xi_{k,l})$
	which by definition is
	$$\xi_{n+m-1,n+m}\cdot\rho^{n+m}_{k+n,l+n} \cdot \ast_{n,m} - \ast_{n-2,m}\cdot ( id\tensor \xi_{m-1,m}\cdot\rho^m_{k,l}).$$
To show this element is in $\op{R}(n,m;n+m-2)$ it suffices to observe that $\rho^{n+m}_{k+n,l+n}=id_n\times \rho^m_{k,l}\in S_{n}\times S_m\subset S_{n+m}$.  Thus we conclude $\op{R}(n,m;n+m-2)= span\{B\}.$ 

 Finally, to show that $B$ is linearly independent, we again use Lemma $\ref{purelem}$ to reduce to the case of pure trees, and then observe that each term appearing in such an element of $B$ appears in no other.

{\bf Statement $(3)$:}  This statement is a bit easier.  Let $\vec{v}=(n_1,n_2,n_3; n_1+n_2+n_3)$.  Since every putative basis element is just a permutation of an element in $\op{R}$, it must be the case that $span\{B(\vec{v})\}=\op{R}(\vec{v})$.  Linear independence can again be established by applying Lemma $\ref{purelem}$, after which one is reduced to considering the vector space $\op{T}_{\ast,\xi}(\vec{v})$ which is three dimensional with basis corresponding to the three non-planar binary trees on three leaves.  
The relations are 2 dimensional, and the quotient is one dimensional (a corolla).  Since two such relations (corresponding to any fixed $\sigma$) are not scalar multiples of each other, linear independence follows. \end{proof}

\subsection{The underlying operad in sets, $\mathsf{SM}$.}

The linear operad $\op{SM}$ is the span of an operad in sets.  Specifically, define an equivalence relation $\sim$ on each set $F(\mathsf{E})(\vec{v})$ by declaring two elements to be equivalent iff their images coincide under the composite of
\begin{equation}\label{sm}
F(\mathsf{E})(\vec{v})\hookrightarrow F(\op{E})(\vec{v}) \twoheadrightarrow \op{SM}(\vec{v}).
\end{equation}
Define $\mathsf{SM}(\vec{v}) $ to be the set of equivalence classes $F(\mathsf{E})(\vec{v})/\sim$.
\begin{lemma}\label{setslem}  The sets $\mathsf{SM}(\vec{v})$ inherit the structure of a $\mathbb{V}$-colored operad for which $\text{span}(\mathsf{SM})= \op{SM}$. 
\end{lemma}
\begin{proof} 
	Since the maps in Equation $\ref{sm}$ are equivariant, the symmetric group actions descend to the quotient making $\mathsf{SM}$ a $\mathbb{V}$-colored sequence.  The operad structure on $\mathsf{SM}$ is defined on representatives by pulling back to $F(\mathsf{E})$.  In particular, it is the unique structure for which $F(\mathsf{E})\to \mathsf{SM}$ is an operad map.	Taking the span of this map, we realize both  $\text{span}(\mathsf{SM})$ and $\op{SM}$ as equivalence classes of elements of
	$\text{span}(F(\mathsf{E})) = F(\op{E})$.  Equivalent elements in the former relation are also equivalent in the latter, so there is canonical map $\text{span}(\mathsf{SM})\to\op{SM}$ each of whose constituents is surjective.

The above paragraph could be applied to any quadratic operad defined by generators forming a $\mathbb{V}$-colored sequence in sets.  The fact that the above map is injective is a further condition which depends on the form of the relations.  Specifically, we apply Proposition $\ref{basis}$ to show that each $\op{R}(\vec{v})$ is contained within the kernel of the map $\text{span}(F(\mathsf{E})\twoheadrightarrow\mathsf{SM})$, since each element of the given bases for $\op{R}$ is of the form $x-y$ for some $x,y \in F(\mathsf{E})$.  It follows that the ideal generated by $\op{R}$ is also in the kernel, hence the induced map $\text{span}(\mathsf{SM})\to \op{SM}$ is also injective. \end{proof}

\subsection{Pure relations}  

The previous lemma confirms that the equivalence relation on $F(\mathsf{E})$ whose equivalence classes are $\mathsf{SM}$ is induced by the equivalences corresponding to the basis elements of Proposition $\ref{basis}$ along with operadic composition in $F(\mathsf{E})$.  To conclude this section let us unpack this statement to give a family of equivalences which generate this equivalence relation for each set $F(\mathsf{E})(\vec{v})$.

Using Lemma $\ref{purelem}$, we will write these equivalences as  $(\sigma,\mathsf{t})\sim (\sigma,\mathsf{t}_e)$ where $\mathsf{t}_e$ is formed from the pure tree $\mathsf{t}$ by applying the relevant relation at the edge $e$.  To make this explicit, fix an edge $e$ of a pure tree $\mathsf{t}$.  Let $u$ denote the vertex below $e$ and let $w$ denote the vertex above $e$. We consider three cases.

\begin{enumerate}
	\item Case $|u|=|w|=1$. 
	In this case the vertex $w$ is labeled by some $\xi_{i,j}$ and the vertex $u$ is labeled by some $\xi_{k^\prime, l^\prime}$.  Define $\mathsf{t_e}$ to be the same as $\mathsf{t}$, except $\xi_{k,l}$ replaces $\xi_{i,j}$ and $\xi_{i^\prime,j^\prime}$ replaces $\xi_{k^\prime,l^\prime}$, where the notation is as in Lemma $\ref{rholem}$.

	\item Case: $|u|=2$ and $|w|=1$.   In this case the vertex $u$ is labeled by some $[\ast_{n,m}]$, where branch $n$ is less than branch $m$ (in the terminology of subsection $\ref{treesec}$) and the vertex $v$ is labeled by some $\xi_{i,j}$.  We first form the tree underlying $\mathsf{t}_e$ by detaching the branch at $u$ not containing $w$, and reattaching it at $w$.  In so doing, the tree underlying $\mathsf{t}_e$ has a binary vertex (formerly $v$) above a unary vertex (formerly $u$).  We then relabel these vertices as follows.  There are two subcases depending on if $w$ was on branch $n$ or $m$.  For the former, relabel the vertices 
	 $\xi_{i,j}$ below $\ast_{n+2,m}$.  For the latter, relabel the vertices $\xi_{k+n,l+n}$ below $\ast_{n,m+2}$. 
	
	\item $|u|=|w|=2$.  
	In this case $w$ is labeled with some $[\ast_{q,r}]$ and $u$ is labeled with either some $[\ast_{q+r,s}]$ or $[\ast_{p,q+r}]$.  In either case, the three incoming branches are ordered (alphabetically) and we form the tree $\mathsf{t}_e$ by detaching the middle branch in this order from $w$ and reattaching it to $u$, while detaching the root branch from $u$ and reattaching it to $w$, switching the directionality of $e$.  The new lower vertex is then labeled with $[\ast_{q,r+s}]$ or $[\ast_{p+q,r}]$ and the new upper vertex with $[\ast_{r,s}]$ or $[\ast_{p,q}]$ in the respective cases.
\end{enumerate}

Note that the case of $|u|=1$ and $|w|=2$ is handled above by switching the roles of $\mathsf{t}$ and $\mathsf{t_e}$.

\begin{lemma} \label{equiv} The equivalence relation generated by $(\sigma, \mathsf{t})\sim (\sigma, \mathsf{t}_e)$ over all $\mathsf{t}$ and edges $e\in \mathsf{t}$ coincides with the equivalence relation defining $\mathsf{SM}$ above.
\end{lemma}
\begin{proof} The equivalence relation generated by $(\sigma, \mathsf{t})\sim (\sigma, \mathsf{t}_e)$ is clearly coarser than the one defining $\mathsf{SM}$.  To show that it is also finer, and hence the same, one easily shows
$$	
	\op{R}(\vec{v})\subset \langle (\sigma, \mathsf{t}) - (\sigma, \mathsf{t}_e) \ | \ \sigma, \mathsf{t},e  \rangle(\vec{v}) \subset \langle\op{R}\rangle(\vec{v}) \subset F(\op{E})(\vec{v})
$$	
from which it follows that $\langle (\sigma, \mathsf{t}) - (\sigma, \mathsf{t}_e) \ | \ \sigma, \mathsf{t},e  \rangle = \langle\op{R}\rangle $.  This completes the proof since the number of equivalences classes for the respective equivalence relations is manifest as the dimension of the quotient space of $F(\op{E})(\vec{v})$ by the respective ideals, hence coincides. \end{proof}

\section{$\op{S}\circ\op{M} = \op{SM}$}

The operad $\op{SM}$ has a weight grading which counts the number of generators.  Moreover, it has a bigrading which counts the number of generators of each type.  This bigrading is induced by a corresponding partition of the operad $\mathsf{SM}$ as we now indicate.

Let $\mathsf{E}$ be as in Equation $\ref{generators}$.  Define $F(\mathsf{E})^{s,m}$ to be the sub-$\mathbb{V}$ colored sequence of $F(\mathsf{E})$ consisting of those trees having $s$ unary vertices and $m$ binary vertices.  The bi-index $(s,m)$ is preserved by the equivalence relation of Lemma $\ref{equiv}$ hence descends to $\mathsf{SM}$.  The bi-index is bi-additive under the operadic composition maps
hence there are suboperads $\mathsf{S} := \mathsf{SM}^{-,0}$ and $\mathsf{M} := \mathsf{SM}^{0,-}$ consisting of those elements represented by compositions of only self-gluings (resp.\ only mergers).

Define $\op{S}$ and $\op{M}$ to be the span of $\mathsf{S}$ and $\mathsf{M}$ respectively.

\subsection{$\op{S}$ and $\op{M}$ are Koszul}

Both $\op{S}$ and $\op{M}$ are weight graded quadratic operads.  The generators of these operads satisfy
$$\mathsf{E_S}(n,n-2) = \{\sigma\xi_{i,j} \ | \ \sigma \in S_{n-2} \text{ and } 1\leq i < j \leq n \} \text{ and } \mathsf{E_M}(n,m;n+m) = S_{n+m}\times \ast_{n,m}$$
and are $0$ elsewhere.  The non-zero relations of $\op{S}$ (resp. $\op{M}$) are those indicated in and above Lemma $\ref{equiv}$ of type (1) (resp.\ type (3)).

\begin{lemma}  $\op{S}$ is Koszul.	
\end{lemma}
\begin{proof}  This follows from the main theorem of \cite{KFC}, as the Feynman category corresponding to $\op{S}$ is cubical, hence Koszul.  
\end{proof}

\begin{lemma}  $\op{M}$ is Koszul.
\end{lemma}
\begin{proof}
	Consider $\mathsf{B}(\op{M})(\vec{v})$.  Assume $v_1+...+v_n =v_0$, since otherwise this chain complex would be zero. 
	In this case, there is an isomorphism of chain complexes
	$$\mathsf{B}(\op{M})(\vec{v})\cong k[S_{v_0}]\tensor  B(Com(n)),$$ where $B$ denotes the uncolored operadic bar construction.  This follows by applying Lemma $\ref{purelem}$ in the special case that the tree has only binary vertices.  	To complete the proof, appeal to the Koszulity of Com. \end{proof}

\subsection{A rewriting rule}
We define a morphism of $\mathbb{V}$-colored sequences
\begin{equation} \label{rewrite}
\lambda\colon \mathsf{E_M}\circ_{(1)} \mathsf{E_S} \to \mathsf{E_S}\circ_{(1)}\mathsf{E_M}
\end{equation}
where $\circ_{(1)}$ means a direct sum over all such trees with one edge. 
Given necessary compatibility of the $S_2$ action with the $\circ_i$ maps, it's sufficient to define maps
$$
\mathsf{E_M}(n-2,m;n+m-2)\circ_1 \mathsf{E_S}(n,n-2)
\stackrel{\lambda_{n,m}}\longrightarrow \mathsf{E_S}(n+m;n+m-2)\circ_1 \mathsf{E_M}(n,m;n+m)
$$
We define $\lambda(\ast_{n-2,m}\circ_1\xi_{i,j}) = \xi_{i,j}\circ_1\ast_{n,m}$ and extend $S_{n+m-2}$ equivariantly. From the $S_2$ equivariance one can derive the formula
$\lambda(\ast_{n,m-2}\circ_2\xi_{i,j}) = \xi_{i+n,j+n}\circ_1\ast_{n,m}$.

\begin{lemma}  $\op{SM}=\op{S}\vee_\lambda\op{M}$
\end{lemma}
\begin{proof}  Simply by inspection (see Proposition $\ref{basis}$), the generators and relations of these quadratic operads coincide. \end{proof}

\subsection{$\op{SM}$ satisfies the diamond lemma}


\begin{theorem}\label{main}  The rewriting rule $\lambda$ (Equation $\ref{rewrite}$) satisfies the Diamond Lemma (Lemma $\ref{diamondlem}$).  In particular $\op{SM}$ is Koszul and $\op{SM} \cong \op{S}\circ\op{M}$.
\end{theorem}
\begin{proof}  To apply the Diamond Lemma we need to show that for each $\vec{v}$, the restriction of the natural surjection
	\begin{equation}\label{inj}
		\op{S}\circ\op{M}(\vec{v}) \to \op{S}\vee_\lambda\op{M}(\vec{v})	
	\end{equation}
 to total weight 3 is injective.  This map preserves the bigrading, so it is enough to check injectivity for each bi-index $(s,m)$ with $s+m=3$.  Note that a vector space $\op{S}\circ\op{M}^{(x,3-x)}(\vec{v})$ is not zero if and only if one of the following holds:
\begin{itemize}
	\item $(s,m)=(3,0)$ and $\vec{n}= (n;n-6)$,
	\item $(s,m)=(2,1)$ and $\vec{n}=(n,m; n+m-4)$,
	\item $(s,m)=(1,2)$ and $\vec{n}=(n,m,l;n+m+l-2)$, or
	\item $(s,m)=(0,3)$ and $\vec{n}= (n,m,l,p;n+m+l+p)$.
\end{itemize}

Next, observe that the map in Equation $\ref{inj}$ is in the image of span, viewed as a functor from the category of operads in sets. To see this, first recall that from Lemma $\ref{setslem}$ we have  $\op{SM}=\text{span}(\mathsf{SM})$.  Similarly for the set operads $\mathsf{S}$ and $\mathsf{M}$ whose spans are $\op{S}$ and $\op{M}$ respectively.  Since left adjoints preserves colimits, we have that $\op{S}\circ\op{M} \cong \text{span}(\mathsf{S}\circ\mathsf{M})$.  Finally we observe that the map in Equation $\ref{inj}$ is defined by taking trees to trees, so is defined at the level of bases and hence is in the image of span.

So, to finish the proof it remains to show that the induced map 
\begin{equation}\label{setin}
\mathsf{S}\circ\mathsf{M}^{(s,m)} \to \mathsf{SM}^{(s,m)}
\end{equation}
is injective when $s+m=3$ (i.e.\ in the four cases above).  Since the functor span preserves injections, this will finish the proof.  The cases $(s,m)=(3,0)$ and $(s,m)=(0,3)$ are immediate since the two sides coincide (up to composition with the unit).  So we focus on the remaining two cases.

{\bf Case: $(s,m)=(2,1)$}.  Fix $n_1$ and $n_2$ with $n_1+n_2\geq 4$, so that $\vec{v}=(n_1,n_2; n_0)$ where $n_0=n_1+n_2-4$.

Define a function (of sets)
$$
\text{sh}\colon F(\mathsf{E})(\vec{v})^{(2,1)} \to S_{n_0}\times \mathsf{U}_{n_1+n_2}
$$
called the shadow, as follows.  After Lemma $\ref{purelem}$ we denote elements in $F(\mathsf{E})(\vec{v})$ as pairs $(\sigma,\mathsf{t})$ where $\sigma \in S_{n_0}$.  The tree $\mathsf{t}$ is then a ``pure'' labeled tree, meaning that its lone binary vertex is labeled by some $\ast_{m_1,m_2}$ and its two unary vertices, call them $u$ and $w$, are labeled with some $\xi_{u_1,u_2}$ and $\xi_{w_1,w_2}$ respectively.  Without loss of generality, we assume that the vertex $u$ is not below the vertex $w$  and if the two are parallel (i.e.\ one on branch 1 and the other on branch 2) then $u$ is on 1 and $w$ is on 2. 

Given such a $\mathsf{t}$ with unary vertices $u$ and $w$, we define integers $\chi(u)$ and $\chi(w)$ as follows.  Define $\chi(u)=0$ unless $u$ is on branch 2, in which case $\chi(u)=n_1$  Define $\chi(w)=0$ unless $w$ in branch $2$, in which case we say $\chi(w) = n_1-2$ if $u$ is on branch $1$ and $\chi(w) = n_1$ if $u$ is not on branch $1$. In particular note $\chi(u)\leq \chi(w)$ by our convention.

  We then define the shadow by the formula $$\text{sh}((\sigma, \mathsf{t})) = (\sigma, \phi(u_1+\chi(u),u_2+\chi(u),w_1+\chi(w),w_2+\chi(w))),$$
  where $\phi$ is as in Equation $\ref{phidef}$.

Next, we show that the shadow lifts to a map
$$\text{sh}\colon \mathsf{SM}(\vec{v})^{(2,1)} \to S_{n_0}\times \mathsf{U}_{n_1+n_2}.$$
After Lemma $\ref{rholem}$ it suffices to check that $\text{sh}(\sigma, \mathsf{t}) \sim 
\text{sh}(\sigma, \mathsf{t}_e)$ for any edge $e$ in such a tree.  If the edge is adjacent to two unary vertices, we apply Lemma $\ref{plusn}$ to conclude
$$\text{sh}((\sigma, \mathsf{t}_e)) = (\sigma, \phi\circ\iota(u_1+\chi(u),u_2+\chi(u),w_1+\chi(w),w_2+\chi(w)))$$
and we appeal to the fact that $\phi\circ\iota = \phi$.  Note in this case that $\chi(u)=\chi(w)$, so the ambiguity of which is which in this latter expression is immaterial.

If the edge is adjacent to a binary and a unary vertex, then we may assume, without loss of generality that the unary vertex is above the binary in $\mathsf{t}$ (since their roles are reversed in $\mathsf{t}_e$).  

In this case $\text{sh}((\sigma, \mathsf{t}))=\text{sh}((\sigma, \mathsf{t}_e))$ by direct inspection.  In particular, if the unary vertex is on branch one, then the indices of the self-gluing won't change when passing from $\mathsf{t}$ to $\mathsf{t_e}$.  If the unary vertex is on branch two, then the indices of the self-gluing will increase when being exchanged with the unary vertex.  If the other unary vertex is on branch $1$ then the binary vertex would be labeled by $\ast_{n_1-2,n_2-2}$, hence the indices would increase by $n_1-2$.  Else, the binary vertex would be labeled by either $\ast_{n_1, n_2-2}$ or $\ast_{n_1, n_2-4}$, and so the indices of the unary vertex would increase by $n_1$.  In each case, the change of indices matches the relation (Lemma $\ref{rholem}$).

Finally, we observe that the elements of $\mathsf{S}\circ\mathsf{M}^{(2,1)}$ clearly map to elements with different shadows, hence this map is injective when restricted to the $(2,1)$ summand.

{\bf Case: $(s,m)=(1,2)$}. 

Let $\vec{v}=(n_1,n_2,n_3; n_0)$ where $n_0=n_1+n_2+n_3-2$.  Write  $\mathsf{P}_n$ for the set of pairs of distinct numbers between $1$ and $n$.  We again define a function (of sets)
$$
\text{sh}\colon F(\mathsf{E})(\vec{v})^{(1,2)} \to S_{n_0}\times \mathsf{P}_{n_0+2}
$$
called the shadow, as follows.  We denote an element in the source by a pair $(\sigma,\mathsf{t})$ where $\sigma \in S_{n_0}$ and $\mathsf{t}$ is a pure labeled tree.  Such a $\mathsf{t}$ has a lone unary vertex, and we define $\chi(u)$ to be the sum of all of the left indices of the binary labels for which $u$ is on the greater/second branch.  In otherwords, if $\ast_{m_1,m_2}$ labels a binary vertex below $u$ it contributes $m_1$ to the sum $\chi(u)$ if $u$ is on the second (or $m_2$) branch and it contributes $0$ if $u$ is on the first branch.

We then define the shadow by
$$\text{sh}(\sigma, \mathsf{t}) = (\sigma, (u_1+\chi(u), u_2+\chi(u))).$$

Similarly to the previous case (but easier) one may check that $\text{sh}((\sigma, \mathsf{t}))=\text{sh}((\sigma, \mathsf{t}_e))$ by direct inspection.  Finally, observe that the elements of $\mathsf{S}\circ\mathsf{M}^{(1,2)}$ map to elements with different shadows, finishing the proof. \end{proof}

\section{Odd SMOs and $BV_\infty$ algebras.}

In this section we briefly tie Theorem $\ref{main}$ into a previously studied relationship between SMOs and BV algebra.  The crucial notion is that of an odd SMO (which we called a non-connected $\mathfrak{K}$-modular operad in \cite{KWZ}).  We can define an odd SMO in the present context in terms of a rewriting rule involving $\op{S}^!$ in place of $\op{S}$.  For this purpose, we begin by considering the quadratic operad $\op{S}^!$.

The Feynman category associated to $\op{S}$ is cubical (\cite{KW},\cite{KFC}) and so its quadratic dual is particularly easy to describe.  First, write $\op{S} = F(\op{E_S})/\langle \op{R_S} \rangle $.  The quadratic dual of $\op{S}^!$ is defined to be $$\op{S}^! = F_{odd} (\op{E_S}^\ast) / \langle \op{R_S}^\perp \rangle,$$ where $F_{odd}$ is the free $\mathbb{V}$-colored odd operad, given explicitly by the formula $F_{odd}=\Sigma^{-1}F(\Sigma (-))$.  In particular the free odd operad differs from the free operad in that edges have degree $1$.  In this example, one could bypass the need to invoke odd structures by observing that $\Sigma\op{S}^!$ is again just a $\mathbb{V}$-colored operad.

The $\mathbb{V}$-colored sequence of generators $\op{E_S}$ comes with a basis, namely $\sigma \xi_{i,j}$, after Lemma $\ref{ijlem}$, and so we can identify $\op{E_S}$ with $\op{E_S}^\ast$.  We then use Proposition $\ref{basis}$ to characterize the subspaces $\op{R}_S\subset F(\op{E_S})^{(2)}$.  This result tells us that when $\op{R}_S(\vec{v})$ is not zero, it is a subspace of half the dimension, spanned by vectors of the form $x-\iota(x)$ (for $x$ a basis element as in statement (1) of Proposition $\ref{basis}$).  The orthogonal complement is thus spanned by $x+\iota(x)$, modulo the identification $\op{E_S}\cong \op{E_S}^\ast$.

We therefore see:

\begin{lemma}  The quadratic operad $\op{S}^!$ has generators $\op{E_S}^!\cong \Sigma \op{E_S}$, and we write $\bar{\xi}_{i,j}$ for the image of $\xi_{i,j}$ via this isomorphism.  The relations of $\op{S}^!$ are generated by all those of the form $$\bar\xi_{i,j}\bar\xi_{k^\prime,l^\prime} = -\bar\xi_{k,l}\bar\xi_{i^\prime,j^\prime}.$$
provided $\iota(i,j,k^\prime,l^\prime) =(k,l,i^\prime, j^\prime)$ (as defined in Equation $\ref{iota}$).
\end{lemma}

Note that $\op{E}_{\op{S}^!}$ may be regarded as the span of the set of $\sigma\bar\xi_{ij}$, along with a shift in degree.  Hence we may, analogous to Equation $\ref{rewrite}$ above, define a rewriting rule
$$
\op{E_M}(n-2,m;n+m-2)\circ_1 \op{E}_{\op{S}^!}(n,n-2)
\stackrel{\bar\lambda_{n,m}}\longrightarrow \op{E}_{\op{S}^!}(n+m;n+m-2)\circ_1 \op{E_M}(n,m;n+m)
$$
by defining $\bar\lambda(\ast_{n-2,m}\circ_1\bar{\xi}_{i,j}) = \bar{\xi}_{i,j}\circ_1\ast_{n,m}$ and extending linearly and $S_{n+m-2}$ equivariantly.   The following is an immediate consequence of Theorem $\ref{main}$.
\begin{corollary}  The rewriting rule $\bar{\lambda}$ satisfies the Diamond lemma.  In particular the operad $\op{S}^!\vee_{\bar{\lambda}} \op{M}$ is Koszul and  $\op{S}^!\circ\op{M}\cong \op{S}^!\vee_{\bar{\lambda}} \op{M}$.
\end{corollary}
\begin{proof}  Fix $\vec{v}$.  The vector spaces $(\op{S}\circ\op{M})^{(3)}(\vec{v})$ and $(\op{S}^!\circ\op{M})^{(3)}(\vec{v})$ have the same dimension.  Similarly $(\op{S}^!\vee_{\bar{\lambda}}\op{M})^{(3)}(\vec{v})$ and $(\op{S}\vee_{\lambda}\op{M})^{(3)}(\vec{v})$ have the same dimension.  Theorem $\ref{main}$ implies that $(\op{S}\circ\op{M})^{(3)}(\vec{v})$ and $(\op{S}\vee\op{M})^{(3)}(\vec{v})$ have the same dimension.  Therefore, all four have the same dimension.  Hence any surjection between two such spaces is also an injection, which is the condition of the Diamond lemma.
\end{proof}
In analogy with the above, we define $\op{S}^!\op{M}:=\op{S}^!\vee_{\bar{\lambda}} \op{M}$.
\begin{definition}  An odd Schwarz modular operad (odd SMO for short) is an algebra over $\op{S}^!\op{M}$.
\end{definition}
In particular, odd SMOs are encoded by a Koszul operad. As an application of this fact, let us recall the relationship between odd SMOs and BV algebras.  The Feynman category perspective views an odd SMO as a functor, hence we can take its limit or colimit.  Specifically, if $A$ is an odd SMO, these give us
$$
colim(A)\cong \bigoplus_{n}A(n)_{S_n} \text{ \ \ \ and \ \ \ } lim(A)\cong \prod_{n}A(n)^{S_n}.
$$
We view $colim(A)\subset lim(A)$ after identifying coinvariants with invariants.  With this terminology we can translate the following result from \cite{KWZ}.

\begin{theorem}\label{KWZ}\cite[Theorem 7.5]{KWZ} The limit of an odd SMO carries the structure of a BV algebra, for which the colimit is a BV subalgebra, generalizing the classical odd Lie bracket on the colimit of an operad.
\end{theorem}

The BV operad is the sum of self-gluings and the commutative product is induced by mergers.  Strictly speaking in \cite{KWZ} we focused on the colimit, but the formulas immediately lift to the limit because given a factor of the product, there are only finitely many indices for which a merger or self-gluing could land in said factor.

We combine this with our Koszulity result above to establish the following homotopy invariant analog.  We define an $\infty$-odd-SMO in the expected way, namely as an algebra over the source of the resolution: 
\begin{equation}\label{res}
	D((\op{S}^! \op{M})^!)\stackrel{\sim}\to \op{S}^! \op{M},
\end{equation}
where $D$ denotes the linear dual of the groupoid colored operadic bar construction.  Then we conclude that
\begin{corollary}  An $\infty$-odd-SMO, carries a $BV_\infty$-structure on its colimit and its limit, lifting the structure on homology given in Theorem $\ref{KWZ}$.
\end{corollary}

\begin{proof} 
	
Let $A$ be an $\infty$-odd-SMO.
	In particular, $A$ is a $\mathbb{V}$-module valued in chain complexes,  over a groupoid $\mathbb{V}$ such that for each $v\in\mathbb{V}$ the group $Aut(v)$ is finite.  Hence, the right exact functor of coinvariants coincides with the left exact functor of invariants to establish that, in this case, $H_\ast(A(n))_{S_n} \cong H_\ast(A(n)_{S_n})$ for all $n$.  Taking their direct product (resp.\ sum) over all $n$ we find  $H_\ast((co)lim(A))\cong (co)lim(H_\ast(A))$.
	
	Invoking the quasi-isomorphism in Equation $\ref{res}$, the $\mathbb{V}$-module $H_\ast(A)$ carries the structure of an odd SMO, and hence the structure of a BV algebra on its limit by Theorem $\ref{KWZ}$.  Combining this with the above paragraph, we have a BV structure on $H_\ast((co)lim(A))$, and the result follows from the groupoid colored  homotopy transfer theorem (after eg \cite{WardMP}). \end{proof}

\end{document}